\documentclass[11pt]{article}

\usepackage{amssymb,amsmath,amsfonts,amsthm}
\usepackage{latexsym}
\usepackage{graphics}
\usepackage{indentfirst}

\newtheorem*{thm*}{Theorem}
\newtheorem{thm}{Theorem}
\newtheorem{lem}[thm]{Lemma}
\newtheorem{pro}[thm]{Proposition}

\newtheorem{conj}[thm]{Conjecture}

\newtheorem{ques}[thm]{Question}

\newcommand{\N}{\mathbb{N}}

\begin{document}

\title{On List Equitable Total Colorings of the Generalized Theta Graph}

\author{Jeffrey A. Mudrock\footnotemark[1], Max Marsh\footnotemark[1], and Tim Wagstrom\footnotemark[1]}

\footnotetext[1]{Department of Mathematics, College of Lake County, Grayslake, IL 60030.  E-mail:  {\tt {jmudrock@clcillinois.edu}}}

\maketitle

\begin{abstract}
In 2003 Kostochka, Pelsmajer, and West introduced a list analogue of equitable coloring called equitable choosability.  A $k$-assignment, $L$, for a graph $G$ assigns a list, $L(v)$, of $k$ available colors to each $v \in V(G)$, and an equitable $L$-coloring of $G$ is a proper coloring, $f$, of $G$ such that $f(v) \in L(v)$ for each $v \in V(G)$ and each color class of $f$ has size at most $\lceil |V(G)|/k \rceil$.  In 2018, Kaul, Mudrock, and Pelsmajer subsequently introduced the List Equitable Total Coloring Conjecture which states that if $T$ is a total graph of some simple graph, then $T$ is equitably $k$-choosable for each $k \geq \max \{\chi_\ell(T), \Delta(T)/2 + 2 \}$ where $\Delta(T)$ is the maximum degree of a vertex in $T$ and $\chi_\ell(T)$ is the list chromatic number of $T$.  In this paper we verify the List Equitable Total Coloring Conjecture for subdivisions of stars and the generalized theta graph.

\medskip



\end{abstract}

\section{Introduction}\label{intro}

In this paper all graphs are nonempty, finite, simple graphs unless otherwise noted.  Generally speaking we follow West~\cite{W01} for terminology and notation.  The set of natural numbers is $\N = \{1,2,3, \ldots \}$.  For $m \in \N$, we write $[m]$ for the set $\{1, \ldots, m \}$.  If $G$ is a graph and $S \subseteq V(G)$, we use $G[S]$ for the subgraph of $G$ induced by $S$.  For $v \in V(G)$, we write $d_G(v)$ for the degree of vertex $v$ in the graph $G$ and $\Delta(G)$ for the maximum degree of a vertex in $G$, and we write $N_G(v)$ (resp. $N_G[v]$) for the neighborhood (resp. closed neighborhood) of vertex $v$ in the graph $G$.  If $e, w \in E(G)$ and $v \in V(G)$, we say $e$ and $v$ are incident if $v$ is an endpoint of $e$, and we say $e$ and $w$ are incident if $e$ and $w$ share an endpoint.  Also, $G^k$ denotes the $k^{th}$ power of graph $G$ (i.e. $G^k$ has the same vertex set as $G$ and edges between any two vertices within distance $k$ in $G$). 

\subsection{Total Coloring, Equitable Coloring, and List Coloring}

In this paper we study a conjecture that combines total coloring, equitable coloring, and list coloring.  So, we begin by briefly reviewing these three notions.

Given a graph $G$, in the classic vertex coloring problem we wish to color the elements of $V(G)$ with colors from the set $[m]$ so that adjacent vertices receive different colors, a so-called \emph{proper $m$-coloring}.  The \emph{chromatic number} of $G$, denoted $\chi(G)$, is the smallest $k$ such that a proper $k$-coloring of $G$ exists.  

\subsubsection{Total Coloring}

A \emph{total $m$-coloring} of $G$ is a labeling $f: V(G) \cup E(G) \rightarrow [m]$ where $f(u) \neq f(v)$ whenever $u$ and $v$ are adjacent or incident in $G$.  The \emph{total chromatic number} of a graph $G$, denoted $\chi''(G)$, is the smallest $k$ such that $G$ has a total $k$-coloring.  Clearly, for any graph $G$, $\chi''(G) \geq \Delta(G)+1$.  A famous open problem is the Total Coloring Conjecture (see~\cite{B65, V68}) which states that for any graph $G$, $\chi''(G) \leq \Delta(G)+2$.  See~\cite{L12} for some other applications of total coloring.

It is possible to rephrase total coloring in terms of classic vertex coloring.  Specifically, the \emph{total graph} of graph $G$, $T(G)$, is the graph with vertex set $V(G) \cup E(G)$ and vertices are adjacent in $T(G)$ if and only if the corresponding elements are adjacent or incident in $G$.  Then $G$ has a total $k$-coloring if and only if $T(G)$ has a proper $k$-coloring.  It follows that $\chi''(G) = \chi(T(G))$.

Given a graph $G$, one can construct $T(G)$ in two steps: first subdivide every edge of $G$ to get a new graph $H$, then take its square (i.e. $T(G)=H^2$).  For example, $T(P_m) = P_{2m-1}^2$.

\subsubsection{Equitable Coloring}

The study of equitable coloring began in 1964 with a conjecture of Erd{\"o}s~\cite{E64}, but it was formally introduced by Meyer in the 1970's~\cite{M73}.  An \emph{equitable $k$-coloring} of a graph $G$ is a proper $k$-coloring of $G$, $f$, such that the sizes of the color classes differ by at most one (where a proper $k$-coloring has exactly $k$ color classes).  In an equitable $k$-coloring, the color classes associated with the coloring are each of size $\lceil |V(G)|/k \rceil$ or $\lfloor |V(G)|/k \rfloor$.  We say that a graph $G$ is \emph{equitably $k$-colorable} if there exists an equitable $k$-coloring of $G$.  Many applications of equitable coloring exist, see for example~\cite{JR02, KJ06, P01, T73}.

Unlike classic vertex coloring, increasing the number of colors can make equitable coloring more difficult.  Indeed for any $m \in \N$, $K_{2m+1, 2m+1}$ is equitably $2m$-colorable, but it is not equitably $(2m+1)$-colorable.  In 1970 Hajn\'{a}l and Szemer\'{e}di~\cite{HS70} proved the 1964 conjecture of Erd{\"o}s: every graph $G$ has an equitable $k$-coloring when $k \geq \Delta(G)+1$. 

In 1994 Chen, Lih, and Wu~\cite{CL94} conjectured that the result of  Hajn\'{a}l and Szemer\'{e}di can be improved by 1 for most connected graphs.  Their conjecture is known as the $\Delta$-Equitable Coloring Conjecture, and it is still open.  Formally, the $\Delta$-Equitable Coloring Conjecture states: a connected graph $G$ is equitably $\Delta(G)$-colorable if it is different from $K_m$, $C_{2m+1}$, and $K_{2m+1,2m+1}$.  The $\Delta$-Equitable Coloring Conjecture has been proven true for interval graphs, bipartite graphs, outerplanar graphs, subcubic graphs, certain planar graphs, and several other classes of graphs (see~\cite{CL94, DW18, DZ18, L98, LW96, YZ97}). 

\subsubsection{List Coloring}

List coloring is another variation on classic vertex coloring that was introduced independently by Vizing~\cite{V76} and Erd\H{o}s, Rubin, and Taylor~\cite{ET79} in the 1970's.  In list coloring, we associate with graph $G$ a \emph{list assignment}, $L$, that assigns to each vertex $v \in V(G)$ a list, $L(v)$, of available colors. Graph $G$ is said to be \emph{$L$-colorable} if there exists a proper coloring $f$ of $G$ such that $f(v) \in L(v)$ for each $v \in V(G)$ (we refer to $f$ as a \emph{proper $L$-coloring} of $G$).  A list assignment $L$ is called a \emph{k-assignment} for $G$ if $|L(v)|=k$ for each $v \in V(G)$.  We say $G$ is \emph{k-choosable} if $G$ is $L$-colorable whenever $L$ is a $k$-assignment for $G$.  The \emph{list chromatic number} of $G$, denoted $\chi_\ell(G)$, is the smallest $k$ for which $G$ is $k$-choosable.  Since a $k$-assignment can assign the same $k$ colors to every vertex of a graph, $\chi(G) \leq \chi_\ell(G)$.  

\subsection{Equitable Total Coloring and List Equitable Coloring}

We now briefly review work that has been done on equitably coloring total graphs, and we review a list analogue of equitable coloring.  We will then spend the remainder of the paper focused upon list equitable total coloring.

\subsubsection{Equitable Total Coloring} 

The study of equitable total coloring was initiated by Fu in 1994~\cite{F94}.  Specifically, Fu introduced the Equitable Total Coloring Conjecture which we now state.   

\begin{conj}[Equitable Total Coloring Conjecture~\cite{F94}] \label{conj: ETCC}
For every graph $G$, $T(G)$ has an equitable $k$-coloring for each $k \geq \max\{\chi(T(G)), \Delta(G)+2 \}$.
\end{conj}

The ``$\Delta(G)+2$'' is required because Fu~\cite{F94} found an infinite family of graphs $G$ with $\chi''(G)= \Delta(G)+1$ but $T(G)$ is not equitably $(\Delta(G)+1)$-colorable (cf. Proposition~2.10 in~\cite{F94}).  Note that if the Total Coloring Conjecture is true, we would have $\max\{\chi''(G), \Delta(G)+2 \} = \Delta(G)+2$.  

Fu~\cite{F94} showed that Conjecture~\ref{conj: ETCC} holds for complete bipartite graphs, complete $t$-partite graphs of odd order, trees, and certain split graphs.  Equitable total coloring has also been studied for graphs with maximum degree 3~\cite{W02}, joins of certain graphs~\cite{GM12,GZ09,ZW05}, the Cartesian product of cycles~\cite{CX09}, and the corona product of cubic graphs~\cite{F15}.

\subsubsection{List Equitable Coloring}

In 2003, Kostochka, Pelsmajer, and West introduced a list analogue of equitable coloring called equitable choosability~\cite{KP03}.  Suppose that $L$ is a $k$-assignment for graph $G$.  An \emph{equitable $L$-coloring} of $G$ is a proper $L$-coloring, $f$, of $G$ such that $f$ uses no color more than $\lceil |V(G)|/k \rceil$ times~\footnote{When it comes to equitable choosability, the word equitable indicates that no color is used excessively often.}.  When an equitable $L$-coloring of $G$ exists we say that $G$ is \emph{equitably $L$-colorable}.  Graph $G$ is \emph{equitably $k$-choosable} if $G$ is equitably $L$-colorable whenever $L$ is a $k$-assignment for $G$. 

We now mention a convention used in this paper.  Suppose that $H$ is a subgraph of $G$, and suppose that $L$ is a $k$-assignment for $G$.   When there is an equitable $L'$-coloring of $H$ where $L'$ is the $k$-assignment for $H$ defined by $L'(v)=L(v)$ for each $v \in V(H)$, we say $H$ has an equitable $L$-coloring.  Notice an equitable $L$-coloring of $H$ requires color classes of size at most $\lceil{|V(H)|/k}\rceil$ which may be more restrictive than the bound required for an equitable $L$-coloring of $G$.

It is important to note that, similar to equitable coloring, making the lists larger may make equitable list coloring more difficult.  Indeed, $K_{1,9}$ is equitably 4-choosable, but it is not equitably 5-choosable.  Also, equitable $k$-choosability does not imply equitable $k$-colorability unless $k=2$.  Indeed $K_{1,6}$ is equitably 3-choosable, but it is not equitably 3-colorable (see~\cite{MC19}).  In~\cite{KP03}, there are (perhaps surprising) conjectures that are list analogues of  Hajn\'{a}l and Szemer\'{e}di's result and the $\Delta$-Equitable Coloring Conjecture.

\begin{conj}[\cite{KP03}] \label{conj: KPW1}
Every graph $G$ is equitably $k$-choosable when $k \geq \Delta(G)+1$.
\end{conj}

\begin{conj}[\cite{KP03}] \label{conj: KPW2}
A connected graph $G$ is equitably $k$-choosable for each $k \geq \Delta(G)$ if it is different from $K_m$, $C_{2m+1}$, and $K_{2m+1,2m+1}$.
\end{conj}

In~\cite{KP03} it is shown that Conjectures~\ref{conj: KPW1} and~\ref{conj: KPW2} hold for forests, connected interval graphs, and 2-degenerate graphs with maximum degree at least 5.  It was also shown in~\cite{KP03} that if $G$ is a graph and $k \geq \max\{\Delta(G), |V(G)|/2 \}$, then $G$ is equitably $k$-choosable unless $G$ contains $K_{k+1}$ or is $K_{k,k}$ with $k$ odd in the latter case. Thus, Conjecture~\ref{conj: KPW2} is true for small graphs (at most $2k$ vertices).  Conjectures~\ref{conj: KPW1} and~\ref{conj: KPW2} have also been verified for outerplanar graphs~\cite{ZB10}, powers of paths and cycles~\cite{KM18}, series-parallel graphs~\cite{ZW11}, and certain planar graphs (see~\cite{LB09, ZB08, ZB15}).  In 2013, Kierstead and Kostochka made substantial progress on Conjecture~\ref{conj: KPW1}, as follows.

\begin{thm}[\cite{KK13}] \label{thm: KKresult}
If $G$ is any graph, then $G$ is equitably $k$-choosable whenever
\[
k \geq
\begin{cases}
\Delta(G)+1 & \text{if} \; \Delta(G) \leq 7 \\
\Delta(G) + \frac{\Delta(G)+6}{7} & \text{if } \; 8 \leq \Delta(G) \leq 30 \\
\Delta(G) +   \frac{\Delta(G)}{6} & \text{if } \; \Delta(G) \geq 31.
\end{cases}
\]
\end{thm}

\subsection{List Equitable Total Coloring}

In 2018, Kaul, Pelsmajer, and the first author, began studying the equitable choosability of total graphs~\cite{KM18} which was originally suggested by Nakprasit~\cite{N02}.  Motivated by the Equitable Total Coloring Conjecture (Conjecture~\ref{conj: ETCC}), they introduced the List Equitable Total Coloring Conjecture (LETCC for short).  

\begin{conj}[{\bf List Equitable Total Coloring Conjecture}~\cite{KM18}] \label{conj: LETCC}
For every graph $G$, $T(G)$ is equitably $k$-choosable for each $k \geq \max \{\chi_\ell(T(G)), \Delta(G)+2 \}$.
\end{conj}

Note that since $\Delta(T(G)) = 2 \Delta(G)$, the LETCC is saying something stronger about total graphs than Conjectures~\ref{conj: KPW1} and~\ref{conj: KPW2} when $\Delta(G) > 2$.  Also, Fu's infinite family of graphs $G$ with $\chi''(G)= \Delta(G)+1$ and $T(G)$ is not eqitably $\Delta(G)+1$-colorable also has the property that $\chi_\ell(T(G))=\Delta(G)+1$ and $T(G)$ is not equitably $(\Delta(G)+1)$-choosable.  So the LETCC would be sharp if true.  The LETCC has been verified for all graphs $G$ with $\Delta(G) \leq 2$, stars, double stars, and trees of maximum degree 3 (see~\cite{KM18, M18}).

\subsection{Outline of Results and an Open Question}

In this paper we study list equitable total coloring of generalized theta graphs.  Suppose that $m \in \N$ and $l_1, \ldots, l_m \in \N$ satisfy $l_1 \leq \cdots \leq l_m$. Then, the \emph{generalized theta graph} $\Theta(l_1, \ldots, l_m)$ is the equivalence class of graphs consisting of two vertices joined by internally disjoint paths of lengths $l_1, \ldots, l_m$.  We will assume that $l_2 \geq 2$ when $m \geq 2$ since we will only be considering simple graphs in this paper.\footnote{For the remainder of this paper, whenever we see $\Theta(l_1, \ldots, l_m)$, we will always assume $m, l_1, \ldots, l_m \in \N$ with $l_1 \leq \cdots \leq l_m$ and $l_2 \geq 2$ when $m \geq 2$.}  Studying list equitable total coloring of generalized theta graphs is quite natural as theta graphs and generalized theta graphs have many interesting properties that have been studied by many researchers (see~\cite{BF16, BH01, CM15, ET79, LB16, LB19, SJ18}).  For this paper, we prove a positive answer to the question: Does the LETCC hold for generalized theta graphs?

In order to build up to the generalized theta graph, in Section~\ref{stars} we study list equitable total coloring of subdivisions of stars.  We say that $H$ is a \emph{subdivision of $G$} if $H$ is a graph obtained from $G$ by replacing the edges of $G$ with internally disjoint paths.  In Section~\ref{stars} we prove the following theorem.

\begin{thm} \label{thm: star}
Suppose $G$ is a subdivision of $K_{1,m}$. If $m=1$, $T(G)$ is equitably $k$-choosable whenever $k\geq 3$. Otherwise $T(G)$ is equitably $k$-choosable whenever $k\geq m+1$. 
\end{thm}  

Suppose $G$ is a subdivision of $K_{1,m}$.  It is worth noting that Theorem~\ref{thm: star} is the best result possible since $\chi(T(G)) \geq \max \{3, m+1\}$.  Also, the result of Theorem~\ref{thm: star} is saying something stronger than: the LETCC holds for subdivisions of stars.  This is because the LETCC only says that $T(G)$ should be equitably $k$-choosable for each $k \geq m+2$.

Finally, in Section~\ref{theta} we prove the following.

\begin{thm} \label{thm: theta}
Suppose $G = \Theta(l_1, \ldots, l_m)$, then $T(G)$ is equitably $k$-choosable whenever $k \geq m+2$. 
\end{thm}

So, the LETCC holds for generalized theta graphs.  Notice that in Theorem~\ref{thm: theta}, $T(G)$ is a path square and cycle square when $m$ is 1 and 2 respectively.  Since path squares with at least 3 vertices are not equitably 2-choosable, and all cycle squares with order not divisible by 3 are not equitably 3-choosable (see~\cite{KM18} for further details), one can see that in the case of $m=1,2$, $T(G)$ may not be equitably $(m+1)$-choosable.  The question of whether $T(G)$ is equitably $(m+1)$-choosable when $m \geq 3$ is open.

\begin{ques} \label{ques: m+1}
Suppose $G = \Theta(l_1, \ldots, l_m)$.  If $m \geq 3$, does it follow that $T(G)$ is equitably $(m+1)$-choosable?
\end{ques}

\section{Subdivisions of Stars} \label{stars}
In this section we prove Theorem~\ref{thm: star}.  Suppose that $m \in \N$ and $l_1, \ldots, l_m \in \N$ satisfy $l_1 \leq \cdots \leq l_m$.  Then, we use $B(l_1, \ldots, l_m)$ to denote the equivalence class of subdivisions of $K_{1,m}$ where the edges of $K_{1,m}$ have been replaced with internally disjoint paths of lengths: $l_1, \ldots, l_m$.  \footnote{For the remainder of this paper, whenever we see $B(l_1, \ldots, l_m)$, we will always assume $m,l_1, \ldots, l_m \in \N$ with $l_1 \leq \cdots \leq l_m$.}

For the remainder of this paper when $G = B(l_1, \ldots, l_m)$, we assume that $V(G) = \{u\} \cup \{v_{i,j}: i \in [m], j \in [l_i]\}$, and the edges of $G$ are drawn so that for each $i \in [m]$ vertices are adjacent if and only if they appear consecutively in the ordering: $u, v_{i,1}, \ldots , v_{i,l_{m}}$.  

Note that when $G = B(l_1, \ldots, l_m)$, $T(G)$ is isomorphic to a copy of $[B(2l_1, \ldots, 2l_m)]^2$.  So, in order to prove Theorem~\ref{thm: star}, we begin by proving the following result which will imply Theorem~\ref{thm: star} for each $m \geq 3$.

\begin{thm} \label{thm: B-graph}
For $m \geq 3$, $(B(l_1,\ldots,l_m))^2$ is equitably $k$-choosable for each $k \geq m+1$.
\end{thm}

Notice that for $m \geq 3$ Theorem~\ref{thm: B-graph} is saying something stronger than Theorem~\ref{thm: star} since we are allowing any of the natural numbers $l_1,\ldots,l_m$ to be odd.  We now prove a Lemma that is closely related to a Lemma appearing in~\cite{KP03}; we use this Lemma frequently to prove our results.  

\begin{lem} \label{lem: KPW}
Let $G$ be a graph and let $L$ be a $k$-assignment for $G$. Suppose that $|V(G)|=kq+r$ where $1\leq r \leq k$.  Suppose $t$ satisfies $r \leq t \leq k$. Let $S=\{x_1, \ldots , x_t\}$ be a set of $t$ distinct vertices in $G$. If $G-S$ has an equitable $L$-coloring and $$|N_G(x_i)-S|\leq k-i$$ for $1 \leq i \leq t$, then $G$ has an equitable $L$-coloring.
\end{lem}

\begin{proof}
Suppose that $f$ is an equitable $L$-coloring of $G-S$ (notice $G-S$ could be the empty graph). Note that no color is used more than $q$ times by $f$.  In an equitable $L$-coloring of $G$ we must use no color more than $\lceil (kq+r)/k \rceil = q+1$ times. Let $L'(x_i) = L(x_i)-\{f(v) : v \in N_G(x_i)-S \}$ for each $i \in [t]$. Since $|N_G(x_i) -S| \leq k-i$, we know that $|L'(x_i)| \geq i$. So, there is a proper $L'$-coloring of $G[S]$ that uses $t$ distinct colors.  Such a coloring along with $f$ completes an equitable $L$-coloring of $G$.
\end{proof}

We will now prove five lemmas that will imply Theorem~\ref{thm: B-graph}.  The first two of the five lemmas will take care of the case where $k \geq m+2$, and the last three of the five lemmas will deal with $k= m+1$.

\begin{lem} \label{lem: m+3}
Suppose $m \geq 3$.  If $H = B(l_1, l_2,\ldots,l_m)$ and $G = H^2$, then $G$ is equitably $k$-choosable whenever $k \geq m+3$.
\end{lem}

\begin{proof}
The result is obvious when $k \geq |V(G)| = 1+\sum^m_{i=1} l_i$. So, we may assume that $L$ is an arbitrary $k$-assignment for $G$ such that $m+3 \leq k < 1+\sum ^m _{i=1}l_i$.  We will show that $G$ is equitably $L$-colorable.  

Since $1+\sum ^m _{i=1}l_i > m+3$, $l_{m} > 1$. Let $S_0 = \{v_{i,1} : i \in [m]\} \cup A$ where $A = \{u, v_{m,2} \}$ if $v_{m,3} \notin V(G)$ and $A = \{u, v_{m,2}, v_{m,3} \}$ otherwise. Note $m+2 \leq |S_0| \leq m+3$.  Let $d=k-|S_0|$.  Let $S_1$ be an arbitrary subset of $V(G)-S_0$ of size $d$, and let $S=S_0 \cup S_1$. Note that $G-S$ is a graph with maximum degree at most 4. By Theorem~\ref{thm: KKresult}, there is an equitable $L$-coloring of $G-S$. 

Now, let $x_1 = u$, $x_{k-3} = v_{m-2,1}$, $x_{k-2} = v_{m-1,1}$, $x_{k-1} = v_{m,2}$, and $x_k = v_{m,1}$.  We then arbitrarily name the remaining vertices in $S$: $x_2, \ldots , x_{k-4}$ in an injective fashion. By the way $S$ is constructed, $|N_G(x_{k-i}) - S| \leq 2$ for $i=2,3$, $|N_G(x_{k-1}) - S| \leq 1$, and $|N_G(x_k) - S| = 0$.  Moreover, $|N_G(x_1) - S| \leq m-1 \leq (m+3) - 1 \leq k-1$.  Finally, for $2 \leq i \leq k-4$, $|N_G(x_i) - S| \leq 4 \leq k-i$. So, Lemma~\ref{lem: KPW} implies $G$ is equitably $L$-colorable.
\end{proof} 

\begin{lem} \label{lem: m+2}
Suppose $m \geq 3$.  If $H = B (l_1, l_2,\ldots,l_m)$ and $G=H^2$, then $G$ is equitably $(m+2)$-choosable.
\end{lem}

\begin{proof} 
Note that the result is obvious when $l_m=1$.  So, we assume that $l_m > 1$.  Let $L$ be an arbitrary $(m+2)$-assignment for $G$. We will show that $G$ is equitably $L$-colorable in the following cases: $l_m=2$, $l_m=3$, and $l_m\geq 4$. 

In the case $l_m=2$, let $S= \{v_{i,1} : i \in [m]\}\cup \{u, v_{m,2}\}$. We name the vertices of $S$ as $x_1, x_2, \ldots , x_{m+2} $ where $x_1=u$, $x_{m+2}=v_{m,2}$, and for each $j \in [m]$, $x_{j+1}=v_{j, 1}$. By Theorem~\ref{thm: KKresult}, $G-S$ is equitably $L$-colorable. It is easy to see that $|N_G(x_1)-S| = m-1 \leq (m+2)-1$, $|N_G(x_{m+2})-S| =0$, $|N_G(x_{m+1})-S| =0$, and for each $2\leq j \leq m$, $|N_G(x_{j})-S| = 1 \leq m+2-j$ for each $j \in [m]$. Thus, $G$ is equitably $L$-colorable by Lemma~\ref{lem: KPW}. 

In the case $l_m=3$, let $S=\{v_{i,1} : 2 \leq i \leq m\} \cup \{u, v_{m,2}, v_{m,3}\}$. We name the vertices of $S$ as $x_1, x_2, \ldots , x_{m+2} $ where $x_1=u$, $x_{m+2}=v_{m,3}$, $x_{m+1}=v_{m,2}$, and for each $2 \leq i \leq m$, $x_{i}=v_{i, 1}$. By Theorem~\ref{thm: KKresult}, $G-S$ is equitably $L$-colorable. It is easy to see that $|N_G(x_1)-S| = m \leq (m+2)-1$, $|N_G(x_{m+2})-S| =0$, $|N_G(x_{m+1})-S| =0$, $|N_G(x_{m})-S| =1$, and for each $2\leq j \leq m-1$, $|N_G(x_{j})-S| \leq 3 \leq m+2-j$. Thus, $G$ is equitably $L$-colorable by Lemma~\ref{lem: KPW}. 

In the case $l_m \geq 4$, let $S =  \{v_{i,1} : 3 \leq i \leq m\} \cup \{u, v_{m,2}, v_{m,3}, v_{m,4}\}$. We name the vertices of $S$ as $x_1, x_2, \ldots , x_{m+2} $ where $x_1=u$, $x_{m+2}=v_{m,2}$, $x_{m+1}=v_{m,3}$, $x_{m}=v_{m,4}$, and $x_{m-1}=v_{m,1}$. Finally, if $m \geq 4$, then for each $2 \leq j \leq m-2$ we let $x_j=v_{j+1,1}$. Notice $G-S$ has maximum degree at most 4.  So, by Theorem~\ref{thm: KKresult}, $G-S$ is equitably $L$-colorable. It is easy to see that $|N_G(x_1)-S|=m+1 \leq (m+2)-1$, $|N_G(x_{m+2})-S|=0$, $|N_G(x_{m+1})-S|\leq 1$, $|N_G(x_{m})-S|\leq 2$, and $|N_G(x_{m-1})-S|=2$. Finally, if $m \geq 4$ then for each $2\leq j \leq m-2$, $|N_G(x_j)-S|\leq 4 \leq m+2-j$. Thus, $G$ is equitably $L$-colorable by Lemma~\ref{lem: KPW}. 
\end{proof}

We now turn our attention to the case of $k=m+1$.  Notice that in the case when $m=3$,  when we try to use Lemma~\ref{lem: KPW}, we will no longer be able to use Theorem~\ref{thm: KKresult} to show $G-S$ is equitably $L$-colorable.  So, we need a result from~\cite{KM18}. 

\begin{pro} [\cite{KM18}] \label{cor: pathsquare}
 For p, n $\in \mathbb{N}$, $P_n^p$ is equitably $k$-choosable whenever $k \geq p+1$.
\end{pro}

Notice that Proposition~\ref{cor: pathsquare} immediately implies that if $G$ is a spanning subgraph of a path square, then $G$ is equitably $k$-choosable whenever $k \geq 3$.

\begin{lem} \label{lem: star}
Suppose $m \geq 3$.  If $H =  B(l_1, l_2,\ldots,l_m)$, $G=H^2$, and $l_1 \leq l_2 \leq 2$, then $G$ is equitably $(m+1)$-choosable.
\end{lem}

\begin{proof}
We may assume that $|V(G)| > m+1$.  Suppose $L$ is an arbitrary $(m+1)$-assignment for $G$. In the case that $l_2 =1$, let $S=\{ u, v_{1,1}, v_{2,1}, \ldots, v_{m,1} \}$. Clearly $G-S$ has an equitable $L$-coloring by Proposition~\ref{cor: pathsquare}. Let $x_1 =u$, and let $x_{i+1}= v_{m+1-i,1}$ for each $i\in[m]$. Note that $|N_G(x_1)-S| \leq m-2 \leq (m+1)-1$, $|N_G(x_i)-S| \leq 2 \leq m+1-i$ for all $2 \leq i \leq m-1$, and $|V_G(x_i)-S| = 0 \leq m+1-i$ when $i=m,m+1$. By Lemma~\ref{lem: KPW}, we know that $G$ has an equitable $L$-coloring.

In the case that $l_2 =2$, let $S=\{u, v_{1,1}, v_{2,1}, \ldots, v_{m-1,1}, v_{2,2}\}$. Clearly $G-S$ has an equitable $L$-coloring by Proposition~\ref{cor: pathsquare}. Let $x_1=u$, $x_{m-1} = v_{1,1}$, $x_m = v_{2,1}$, and $x_{m+1}=v_{2,2}$.  If $m \geq 4$, let $x_{i+1}=v_{m-i,1}$ for each $i \in [m-3]$. Note: $|N_G(x_1)-S| \leq m+1-1$, $|N_G(x_{m-1})-S| \leq 2$, $|N_G(x_m)-S| = 1 $, and $|N_G(x_{m+1})-S| = 0$. Furthermore, if $m \geq 4 $, for each $2 \leq i \leq m-2$, $|N_G(x_i)-S| \leq 3 \leq m+1-i$.  By Lemma~\ref{lem: KPW}, we know that $G$ has an equitable $L$-coloring.
\end{proof}

We now would like to prove for $m \geq 3$ that if $H = B (l_1, l_2,\ldots,l_m)$ and $G=H^2$, then $G$ is equitably $(m+1)$-choosable by induction on $\sum_{i=1}^m l_i$.  Lemma~\ref{lem: star} takes care of the base case.  The next Lemma takes care of a small issue in the inductive step when $m=3$.

\begin{lem} \label{lem: problemcase}
If $H = B(1,3,3)$ and $G=H^2$, then $G$ is equitably $4$-choosable.
\end{lem}

\begin{proof}
Suppose that $L$ is an arbitrary 4-assignment for $G$. Let $S = \{v_{2,3}, v_{3,3}, v_{3,2}, v_{3,1} \}$. Note that $G-S$ is the square of a path. So, by Proposition~\ref{cor: pathsquare} we know that $G-S$ has an equitable $L$-coloring. We then let $x_1 = v_{3,1}$, $x_2 =v_{2,3}$, $x_3 = v_{3,2}$, and $x_4=v_{3,3}$. Note that $|N_G(x_i)-S| = 4-i$ for all $i \in [4]$. So, by Lemma~\ref{lem: KPW}, we know that $G$ has an equitable $L$-coloring.
\end{proof}

We are finally ready to complete our proof of Theorem~\ref{thm: B-graph}.

\begin{lem} \label {lem: m+1}
Suppose that $m \geq 3$.  If $H = B (l_1, l_2,\ldots,l_m)$ and $G=H^2$, then $G$ is equitably $(m+1)$-choosable.
\end{lem}

\begin {proof}
We will prove the desired by induction on $\sum_{i=1}^m l_i$. Let $C= \sum_{i=1}^m l_i$. Let $L$ be an arbitrary $(m+1)$-assignment for $G$. We will show an equitable $L$-coloring of $G$ exists for each $C \geq m$. For the base case suppose that $m \leq C \leq 3m-3$. Since $C \leq 3m-3$, we know that $l_1 \leq l_2 \leq 2$. So, the desired result holds by Lemma~\ref{lem: star}. 

For the inductive step suppose that $C \geq 3m-2$ and assume that the desired result holds for all natural numbers less than $C$ and at least $m$. Note that when $l_2 \leq 2$ the result holds by Lemma~\ref{lem: star}; so, we may assume that $l_2 \geq 3$. 

If $l_m \geq 4$ we let $S= \{v_{j,l_j}: 2\leq j \leq m-1 \} \cup \{v_{m,l_m-2},v_{m,l_m-1}, v_{m,l_m}\}$. By the inductive hypothesis, $G-S$ is equitably $L$-colorable. Now let $x_i=v_{i+1,l_{i+1}}$ for all $i \in [m-2]$, $x_{m-1} = v_{m,l_m-2}$, $x_m=v_{m,l_m-1}$, and $x_{m+1} = v_{m,l_m}$. Note that $|N_G(x_i) -S| =2 \leq (m+1)-i$ for all $i \in [m-1]$, $|N_G(x_m) -S| =1$, and  $|N_G(x_{m+1})-S|=0$. Thus, Lemma~\ref{lem: KPW} implies that $G$ is equitably $L$-colorable.

Now, assume that $l_m \leq 3$ (note that this means $l_m = 3$ since $l_2 \geq 3$).  This implies that $3m-2 \leq C \leq 3m$.  Let $d=C-(2m+1)$.  In the case that $C=3m-2$, we may assume that $m \geq 4$ by Lemma~\ref{lem: problemcase}. Note that $m-3 \leq d \leq m-1$ and that $|V(G)|-d= C+1-d= 2m+2$. We let $S=\{v_{2,l_2}, \ldots , v_{1+d,l_{1+d}}\}$. By the inductive hypothesis, we know that $G-S$ has an equitable $L$-coloring. Let $x_i = v_{i+1, l_{i+1}}$ for all $i \in [d]$. Note that $|N_G(x_i)-S|= 2 \leq (m+1)-i$ for all $i \in [d]$. Lemma~\ref{lem: KPW} then implies that $G$ has an equitable $L$-coloring.  The induction step is now complete.
\end{proof}

Finally, notice the result of Theorem~\ref{thm: star} is implied by Theorem~\ref{thm: B-graph} when $m \geq 3$, and the result of Theorem~\ref{thm: star} is implied by Proposition~\ref{cor: pathsquare} when $m=1,2$. 

\section{Generalized Theta Graphs} \label{theta}

In this Section we will prove Theorem~\ref{thm: theta}.  Throughout this section if $G=\Theta(l_1, \ldots, l_m)$, we will assume that the vertices that are the common endpoints in $V(G)$ are $u$ and $w$. We also let the vertices of the $i$th path be \footnote{Notice that if $l_1=1$, the first path has no internal vertices.}: 
$$u, v_{i,1}, \ldots , v_{i,l_i-1}, w.$$
When it comes to proving Theorem~\ref{thm: theta}, it is crucial to note that if $G=\Theta(l_1, \ldots, l_m)$, then $T(G)$ is a copy of $[\Theta(2l_1, \ldots, 2l_m)]^2$ where $2l_2 \geq 4$ whenever $m \geq 2$.  So, the results in this Section will focus upon the equitable choosability of the squares of generalized theta graphs with sufficiently long paths.  Notice that if $G=\Theta(l_1, \ldots, l_m)$, $T(G)$ is a path square and cycle square on at least 6 vertices when $m$ is 1 and 2 respectively.  So, when $m=1,2$ the result of Theorem~\ref{thm: theta} is implied by Proposition~\ref{cor: pathsquare} and the following result.

\begin{pro} [\cite{KM18}] \label{pro: cyclesquare}
Suppose that $p , n \in \N$ with $p \geq 2$ and $n \geq 2p+2$.  Then, $C_{n}^p$ is equitably $k$-choosable for each $k \geq 2p$.
\end{pro}

So, to complete the proof of Theorem~\ref{thm: theta}, we may focus our attention on the case where $m \geq 3$.  We begin by proving the following result.

\begin{thm} \label{thm: thetam+3}
For $m \geq 3$, $l_1 \geq 2$, and $l_2 \geq 4$,  $ [\Theta(l_1, l_2,\ldots,l_m)]^2$ is equitably $k$-choosable whenever $k \geq m+3$.
\end{thm}

We will establish two lemmas that will immediately imply Theorem~\ref{thm: thetam+3}.

\begin{lem}\label{lem: theta2m+2}

Suppose that $m \geq 3$.  Suppose $H=\Theta(l_1, l_2, \ldots, l_m)$ where $l_1 \geq 2$ and $l_2 \geq 4$. If $G= H^2$, then $G$ is equitably $k$-choosable whenever $k \geq 2m+2$. 

\end{lem}

\begin{proof}

The result is obvious when $k\geq |V(G)|$; thus, we will assume that $2m+2 \leq k < |V(G)| = 2+ \sum_{i=1}^m (l_{i}-1)$. Suppose $L$ is an arbitrary $k$-assignment for $G$. Let

$$S_0=\{u, w, v_{m,2}, v_{m,3}\} \cup \{v_{i,1}: 1\leq i \leq m\}\cup \{v_{i,l_i-1}: 3 \leq i \leq m\}.$$

Note that $G-S_0$ is a spanning subgraph of a disjoint union of path squares. Let $d=k-|S_0|$. Let $S_1$ be an arbitrary subset of $V(G)-S_0$ of size $d$. Then let $S=S_0 \cup S_1$. Note that $G-S$ has maximum degree at most 4.  By Theorem~\ref{thm: KKresult}, there is an equitable $L$-coloring of $G-S$. We will name the vertices of $S$: $x_1, x_2, \ldots, x_{k}$ such that $x_1=w$, $x_2=u$, $x_{k-3}=v_{2,1}$, $x_{k-2}=v_{m,3}$, $x_{k-1}=v_{m,2}$, $x_k=v_{m,1}$. We then arbitrarily name the remaining vertices in $S$: $x_3, x_4, \ldots , x_{k-4}$ in an injective fashion. For $i \in \{k-3, k-2, k-1, k\}$ we have $|N_G(x_i)-S| \leq k-i$. For $3 \leq i \leq k-4$ we have that $|N_G(x_i)-S| \leq 4 \leq k-i$. Finally, $|N_G(x_2)-S| \leq m-1 \leq k-2$ and $|N_G(x_1)-S| \leq m+2 \leq k-1$. By Lemma~\ref{lem: KPW}, there is an equitable $L$-coloring for $G$.

\end{proof}

\begin{lem}\label{lem: thetam+3}

Suppose that $m \geq 3$.  Suppose $H=\Theta(l_1, l_2, \ldots, l_m)$ where $l_1 \geq 2$ and $l_2 \geq 4$. If $G= H^2$, then $G$ is equitably $k$-choosable whenever $m+3 \leq k \leq 2m+1$.

\end{lem}

\begin{proof}
Suppose that $L$ is an arbitrary $k$-assignment for $G$ such that $m+3 \leq k \leq 2m+1$. We let $S = \{u,v_{m,3}\} \cup \{v_{i,1}: i \in [m]\} \cup \{v_{i,2}: 2m+3-k \leq i \leq m \}$. Note that $|S|=k$. Note that there is an $r \in \{m-2,m-1,m\}$ and natural numbers $a_1, \ldots, a_r$ with $a_1 \leq \cdots \leq a_r$ such that $G-S$ is isomorphic to $[B(a_1,a_2,\ldots, a_r)]^2$. When $r \leq 2$, we know that $G-S$ has an equitable $L$-coloring by Theorem~\ref{thm: KKresult} since $\Delta(G-S) \leq 4$.  When $r \geq 3$, we know that $G-S$ has an equitable $L$-coloring by Lemma~\ref{lem: m+3}. Let $x_1 = v_{m,3}$, $x_2=v_{1,1}$, $x_3 = u$, $x_{k-1}=v_{m,2}$, and $x_k = v_{m,1}$. Finally, we arbitrarily name the remaining vertices in $S$: $x_3, x_4, \ldots , x_{k-2}$ in an injective fashion. Note that $|N_G(x_1)-S| \leq m \leq k-1$, $|N_G(x_2)-S| \leq m \leq k-2$, $|N_G(x_3)-S| \leq m-1 \leq k-3$, $|N_G(x_{k-1})-S| = 1$, and $|N_G(x_k)-S|=0$. Finally note that $|N_G(x_i)-S| \leq 2 \leq k-i$ for all $4 \leq i \leq k-2$. So by Lemma~\ref{lem: KPW} we know that $G$ has an equitable $L$-coloring.
\end{proof}

We are now finished with the proof of Theorem~\ref{thm: thetam+3}.  To complete the proof of Theorem~\ref{thm: theta} we must show that if $G = \Theta(l_1, l_2, \ldots, l_m)$, then $T(G)$ is equitably $k$-choosable when $m \geq 3$ and $k=m+2$.  Recall for $m \geq 3$ that if $G = \Theta(l_1, l_2, \ldots, l_m)$, then $T(G)$ is a copy of $[\Theta(2l_1, 2l_2, \ldots, 2l_m)]^2$ where $2l_2 \geq 4$.  So, we begin working on the case of $k=m+2$ by dealing with $[\Theta(2,4, \ldots, 4)]^2$ and $[\Theta(4, \ldots, 4)]^2$.  We will then finish the case of $k= m+2$ by considering $[\Theta(l_1, l_2, \ldots, l_m)]^2$ with $l_2 \geq 2$, $l_2 \geq	 4$, and $l_m \geq 6$.

We begin with two specific cases to which our general arguments do not apply.

\begin{lem}\label{lem: 3m 3}
Suppose $H=\Theta(2,4,4)$ and $G = H^2$. Then, $G$ is equitably $5$-choosable.
\end{lem}

\begin{proof}
Suppose $L$ is an arbitrary $5$-assignment for $G$. We will show that $G$ has an equitable $L$-coloring in two cases: (1) $L(v)$ is the same list for each $v \in V(G)$ and (2) there exist $x, y \in V(G)$ such that $L(x) \neq L(y)$. In case (1) we may suppose $L(v) = \{1,2,3,4,5\}$ for each $v \in V(G)$. We let $f$ be the proper $L$-coloring of $G$ defined as follows: $f(u)=1$, $f(v_{1,1})=2$, $f(v_{2,1})=3$, $f(v_{3,1})=4$, $f(v_{2,2})=2$, $f(v_{3,2})=3$, $f(v_{2,3})=4$, $f(v_{3,3})=1$, and $f(w)=5$. Clearly, $f$ is an equitable $L$-coloring of $G$. 

We now turn our attention to case (2). Let $S_1 = \{ v_{2,1}, v_{2,2}, v_{2,3}\}$ and $S_2 = \{ v_{3,1}, v_{3,2}, v_{3,3}\}$. Note that $(V(G)-S_1) \cup (V(G)-S_2) = V(G)$ and $(V(G)-S_1) \cap (V(G)-S_2) \neq \emptyset$. So, it must be the case that there exists at least two elements in either $\{ L(v): v \in V(G)-S_1 \}$ or $\{ L(v) : v \in V(G)-S_2 \}$. We assume without loss of generality that there are at least two elements in $\{ L(v): v \in V(G)-S_1 \}$. It is then possible to find a proper $L$-coloring, $f$, of $G-S_1$ that uses six distinct colors.  Let $L'(v_{2,i}) = L(v_{2,i})-\{f(v): v \in N_G(v_{3,i})-S_1\}$ for each $i \in [3]$. Note that $|L'(v_{2,1})| \geq 2$, $|L'(v_{2,2})| \geq 3$, and $|L'(v_{2,3})| \geq 2$. So, there is a proper $L'$-coloring of $G[S_1]$.  Such a coloring completes an equitable $L$-coloring of $G$.
\end{proof}

\begin{lem}\label{lem: 3m 4}
Suppose $H=\Theta(2,4,4,4)$ and $G = H^2$. Then, $G$ is equitably $6$-choosable.
\end{lem}

\begin{proof}
Suppose $L$ is an arbitrary $6$-assignment for $G$.  Note an equitable $L$-coloring of $G$ uses no color more than twice.  We will show that $G$ has an equitable $L$-coloring in two cases: (1) $L(v)$ is the same list for each $v \in V(G) - \{u,w \}$ and (2) there exist $x, y \in V(G) - \{u,w \}$ such that $L(x) \neq L(y)$. In case (1) we may suppose $L(v) = \{1,2,3,4,5, 6\}$ for each $v \in V(G) -\{u,w \}$.  Let $f$ be the proper $L$-coloring for $G - \{u,w\}$ given by: $f(v_{1,1}) = 2$, $f(v_{2,1}) = 3$, $f(v_{3,1}) = 4$, $f(v_{4,1}) = 5$, $f(v_{2,2}) = 2$, $f(v_{3,2}) = 6$, $f(v_{4,2}) = 6$, $f(v_{2,3}) = 1$, $f(v_{3,3}) = 5$, and $f(v_{4,3}) = 4$.  Then, let $L'(u) = L(u) - \{2,3,4,5,6\}$ and $L'(w) = L(w)- \{1,2,4,5,6\}$.  As long as $L'(u)$ and $L'(w)$ are not the same set of size 1, we can find a proper $L'$-coloring of $G[\{u,w\}]$ which along with $f$ completes an equitable $L$-coloring of $G$.  So, we assume that $L'(w) = L'(u) = \{7 \}$.  This means $L(u)=\{2,3,4,5,6,7\}$ and $L(w)=\{1,2,4,5,6,7\}$. An equitable $L$ coloring, $g$, of $G$ can then be constructed as follows: $g(u)= 7$, $g(v_{1,1}) = 2$, $g(v_{2,1}) = 3$, $g(v_{3,1}) = 4$, $g(v_{4,1}) = 5$, $g(v_{2,2}) = 2$, $g(v_{3,2}) = 3$, $g(v_{4,2}) = 6$, $g(v_{2,3}) = 6$, $g(v_{3,3}) = 5$, $g(v_{4,3}) = 4$, and $g(w) = 1$.  

For case (2), let $S_1=\{u,w,v_{4,1},v_{4,2},v_{3,2}\}$, $S_2=\{u,w,v_{3,1},v_{3,2},v_{2,2}\}$, and \\ $S_3=\{u,w,v_{2,1},v_{2,2},v_{4,2}\}$. Note that $(V(G)-S_1) \cup (V(G)-S_2) \cup (V(G)-S_3) = V(G)-\{u,w\}$ and $\bigcap_{i=1}^3 (V(G)-S_i) \neq \emptyset$. Thus, there must exist at least two elements in $\{ L(v) : v \in V(G)-S_1 \}$, $\{ L(v) : v \in V(G)-S_2 \}$, or $\{ L(v) : v \in V(G)-S_3 \}$. Assume without loss of generality that $\{ L(v) : v \in V(G)-S_1 \}$ has at least two elements. It is possible to find a proper $L$-coloring, $h$, of $G-S_1$ that uses seven distinct colors.  Let $L'(v) = L(v)- \{h(x): x \in N_G(v) - S_1\}$ for each $v \in S_1$. Note that $|L'(u)| \geq 2$, $|L'(w)| \geq 1$, $|L'(v_{4,1})| \geq 2$, $|L'(v_{3,2})| \geq 4$, $|L'(v_{4,2})| \geq 5$.  Now, we greedily construct a proper $L'$-coloring of $G[S_1]$, $h'$, that uses at least 4 distinct colors by coloring the vertices of $S_1$ in the order: $w, u, v_{4,1}, v_{3,2}, v_{4,2}$ (this is possible since $v_{4,1}$ is not adjacent to $w$ in $G$).  In the case $|h'(S_1)|=5$, $h$ together with $h'$ forms an equitable $L$-coloring of $G$.  So, we suppose that $|h'(S_1)|=4$. By the way $h'$ is constructed, it must be that $h'(w)=h'(v_{4,1}) = c$.  

Now, for the sake of contradiction, we suppose that $h$ together with $h'$ is not an equitable $L$-coloring of $G$.  This means that $c$ was used by $h$.  Note that $(V(G) - S_1) \subseteq (N_G(w) \cup N_G(v_{4,1}))$.  So, by the definition of $L'$, $c \notin L'(w)$ or $c \notin L'(v_{4,1})$ which is a contradiction. Thus $h$ along with $h'$ forms an equitable $L$-coloring of $G$.
\end{proof}      

We are now ready to prove two Lemmas that will complete the case of $k = m+2$ for $[\Theta(2,4, \ldots, 4)]^2$ and $[\Theta(4, \ldots, 4)]^2$.

\begin{lem}\label{lem: 3m}
Suppose $m \geq 5$.  Suppose $H=\Theta(l_1, l_2, \ldots, l_m)$ where $l_1 = 2$, and $l_2 =l_m = 4$. If $G = H^2$, then $G$ is equitably $(m+2)$-choosable.
\end{lem}

\begin{proof}
Suppose that $L$ is an arbitrary $(m+2)$-assignment for $G$.  We will construct an equitable $L$-coloring of $G$.  Since $m \geq 5$, in an equitable $L$-coloring of $G$, no color can be used more than $\lceil 3m/(m+2) \rceil = 3$ times. Let $S_1= \{u\} \cup \{v_{i,1}: i \in [m]\}$, $S_2 = \{w\} \cup \{v_{i,3}: 2 \leq i \leq m\}$, and $S_3 = \{v_{i,2}: 2 \leq i \leq m \}$. Note that $|S_1| = m+1$, $|S_2|=m$, and $|S_3| =m-1$. We begin by coloring the vertices in $S_1$ with $m+1$ distinct colors. For each $v \in S_2$ suppose that $L'(v)$ is obtained from $L(v)$ by deleting the colors used on the neighbors of $v$ in coloring the vertices in $S_1$. Note that $|L'(v)| \geq m$ for each $v \in S_2$. So, we can color each $v \in S_2$ with a color $c \in L'(v)$ such that the vertices in $S_2$ are colored with $m$ distinct colors.  

Now, for each $v \in S_1 \cup S_2$ let $f(v)$ be the color used to color $v$.  Note that $f$ uses at most $m$ colors twice, and $f$ uses no color more than twice.  For each $v\in S_3$ suppose that $L''(v) = L(v) - \{f(x) : x \in (S_1 \cup S_2) \cap N_G(v) \} $. Since each vertex in $S_3$ has degree 4 in $G$, $|L''(v)| \geq m-2$ for each $v \in S_3$.

If it is not the case that: $|L''(v)| =m-2$ for all $v \in S_3$ and $L''(v_{2,2}) = L''(v_{3,2}) = \cdots = L''(v_{m,2})$, then it is clear that there exists a proper $L''$-coloring of $G[S_3]$ that uses $m-1$ distinct colors which completes an equitable $L$-coloring of $G$. So, we assume that $|L''(v)| =m-2$ for all $v \in S_3$ and $L''(v_{2,2}) = L''(v_{3,2}) = \cdots = L''(v_{m,2})$.  Let $A = L''(v_{2,2})$.  For the sake of contradiction, we assume that all colors in $A$ are used twice by $f$. Note that $f^{-1}(A) \subseteq S_1 \cup S_2$ and $|f^{-1}(A)| = 2|A| = 2m-4$.  Since $|S_1 \cup S_2| =2m+1$, there are at most 5 vertices in $(S_1 \cup S_2) - f^{-1}(A)$. Note that 
$$\left |\bigcup_{i=2}^m N_G(v_{i,2}) \right| = 2m \geq 10 > 5.$$ 
So, there must be a $z \in S_1 \cup S_2$ that is in $f^{-1}(A) \cap \bigcup_{i=2}^m N_G(v_{i,2})$.  This however implies that for some $2 \leq i \leq m$, $f(z)$ was deleted from $L(v_{i,2})$ when forming $L''(v_{i,2})$ which implies $f(z) \notin L''(v_{i,2})$ which is contradiction. 

Thus, there must exist an element $a \in A$ that was not used twice by $f$. So, we can complete an equitable $L$-coloring of $G$ by coloring $v_{2,2}$ and $v_{3,2}$ with $a$ and coloring the remaining vertices in $S_3$ with the $m-3$ distinct colors in $A-\{a\}$. 
 \end{proof}

\begin{lem}\label{lem: 3m+2}
Suppose $m \geq 3$.  Suppose $H=\Theta(l_1, l_2, \ldots, l_m)$ where $l_1 =l_m = 4$. If $G = H^2$, then $G$ is equitably $(m+2)$-choosable.
\end{lem}

\begin{proof}
Suppose that $L$ is an arbitrary $(m+2)$-assignment for $G$.  We will construct an equitable $L$-coloring of $G$.  Since $m \geq 3$, in an equitable $L$-coloring of $G$, no color can be used more than $\lceil (3m+2)/(m+2) \rceil = 3$ times. Let $S_1=\{u,v_{1,2}\} \cup \{v_{i,1}: i \in [m]\}$, $S_2=\{w\}  \cup \{v_{i,3}: i \in [m]\}$, and $S_3 = \{v_{i,2}: 2 \leq i \leq m\}$. Note that $|S_1| = m+2$, $|S_2|=m+1$, and $|S_3|=m-1$.  We begin by coloring the vertices in $S_1$ with $m+2$ distinct colors. For each $v \in S_2$ suppose that $L'(v)$ is obtained from $L(v)$ by deleting the colors used on the neighbors of $v$ in coloring the vertices in $S_1$. Note that $|L'(w)| \geq m+1$ and $|L'(v_{i,3})| \geq m$ for all $i \in [m]$. So, we can color each $v \in S_2$ with a color $c \in L'(v)$ such that the vertices in $S_2$ are colored with $m+1$ distinct colors.  

Now, for each $v \in S_1 \cup S_2$ let $f(v)$ be the color used to color $v$.  Note that $f$ uses at most $m+1$ colors twice, and $f$ uses no color more than twice.  For each $v\in S_3$ suppose that $L''(v) = L(v) - \{f(x) : x \in (S_1 \cup S_2) \cap N_G(v) \} $. Since each vertex in $S_3$ has degree 4 in $G$, $|L''(v)| \geq m-2$ for each $v \in S_3$. 

If it is not the case that: $|L''(v)| =m-2$ for all $v \in S_3$ and $L''(v_{2,2}) = L''(v_{3,2}) = \cdots = L''(v_{m,2})$, then it is clear that there exists a proper $L''$-coloring of $G[S_3]$ that uses $m-1$ distinct colors which completes an equitable $L$-coloring of $G$. So, we assume that $|L''(v)| =m-2$ for all $v \in S_3$ and $L''(v_{2,2}) = L''(v_{3,2}) = \cdots = L''(v_{m,2})$.  Let $A = L''(v_{2,2})$.  For the sake of contradiction, we assume that all colors in $A$ are used twice by $f$.  Note that $f^{-1}(A) \subseteq S_1 \cup S_2$ and $|f^{-1}(A)| = 2|A| = 2m-4$.  Since $|S_1 \cup S_2| =2m+3$, there are at most 7 vertices in $(S_1 \cup S_2) - f^{-1}(A)$. Note that 
$$\left |\bigcup_{i=2}^m N_G(v_{i,2}) \right| = 2m.$$
So, when $m \geq 4$, there must be a $z \in S_1 \cup S_2$ that is in $f^{-1}(A) \cap \bigcup_{i=2}^m N_G(v_{i,2})$, and we reach a contradiction as we did in the proof of Lemma~\ref{lem: 3m}.  When $m=3$ note that since $\{v_{1,1}$, $v_{1,2}$, $v_{1,3} \}$ is a clique in $G$, the single element in $A$ must be in $f(\bigcup_{i=2}^3 N_G(v_{i,2}))$.  So, when $m=3$ there must also  be a $z \in S_1 \cup S_2$ that is in $f^{-1}(A) \cap \bigcup_{i=2}^m N_G(v_{i,2})$, and we reach a contradiction as we did in the proof of Lemma~\ref{lem: 3m}. 

Finally, we can complete an equitable $L$-coloring of $G$ as we did in the proof of Lemma~\ref{lem: 3m}.
\end{proof}

We now complete the proof of Theorem~\ref{thm: theta} with two Lemmas.  The next Lemma will be important for proving the final Lemma which will address all remaining cases needed for Theorem~\ref{thm: theta} 

\begin{lem} \label{lem: extra edge}
Suppose $m \geq 3$.  Suppose $H=B(l_1,l_2,\ldots, l_m)$, where $l_m \geq 3$.  Suppose $G' = H^2$. Let $G$ be the graph obtained from $G'$ by adding an extra edge between the vertices $v_{a,l_a}$ and $v_{b,l_b}$ where $a$ and $b$ are chosen so that $1 \leq a < b \leq m$. Then, $G$ is equitably $(m+2)$-choosable.
\end{lem}

Note that this Lemma can be extended to $l_m =2$, but this is not necessary to complete the proof of Theorem~\ref{thm: theta}. 

\begin{proof}
Let $L$ be an arbitrary $(m+2)$-assignment for $G$. We will show that $G$ has an equitable $L$-coloring in the following cases: (1) $l_m=3$ and (2) $l_m\geq 4$.

In the case $l_m=3$, let $S=\{u\}\cup \{v_{i,1} : 2 \leq i \leq m\} \cup \{v_{m,2}, v_{m,3}\}$. Then, we name the vertices of $S$: $x_1, x_2, \ldots , x_{m+2}$ where $x_1=u$, $x_{m+2}=v_{m,2}$, $x_{m+1}=v_{m,3}$, and for each $2 \leq i \leq m$, $x_{i}=v_{i, 1}$. Note $\Delta (G-S) \leq 3$.  So, by Theorem~\ref{thm: KKresult}, $G-S$ is equitably $L$-colorable. It is easy to see that $|N_G(x_1)-S| = m \leq m+2-1$, $|N_G(x_{m+2})-S| =0$, $|N_G(x_{m+1})-S| \leq 1$, $|N_G(x_{m})-S| =1$, and for each $2\leq j \leq m-1$, $|N_G(x_{j})-S| \leq 3 \leq m+2-j$. Thus $G$ is equitably $L$-colorable by Lemma~\ref{lem: KPW}. 

For case (2), we will show that $G$ has an equitable $L$-coloring in the following sub-cases: (a) $m \geq 4$ and (b) $m=3$.  When $m\geq 4$, choose a $t$ such that $t \in [m]$, $t \neq a$, $t \neq b$, and $t \neq m$. Let 
$$S = \{u, v_{m,2}, v_{m,3}, v_{m,4}\} \cup (\{v_{i,1} : i \in [m]\}-\{v_{a,1},v_{t,1}\}).$$ 
Then we name the vertices of $S$: $x_1, x_2, \ldots , x_{m+2} $ where $x_1=u$, $x_{m+2}=v_{m,2}$, $x_{m+1}=v_{m,3}$, $x_{m}=v_{m,4}$, and $x_{m-1}=v_{m,1}$. Finally, we arbitrarily name the remaining vertices in $S$: $x_2, \ldots , x_{m-2}$ in an injective fashion. Note $\Delta(G-S) \leq 4$, and so by Theorem~\ref{thm: KKresult}, $G-S$ is equitably $L$-colorable. It is easy to see that $|N_G(x_1)-S|=m+1 \leq m+2-1$, $|N_G(x_{m+2})-S|=0$, $|N_G(x_{m+1})-S|\leq 1$, $|N_G(x_{m})-S|\leq 2$, and $|N_G(x_{m-1})-S|=2$. Finally, for each $2\leq j \leq m-2$, $|N_G(x_j)-S|\leq 4\leq m+2-j$. Thus, $G$ is equitably $L$-colorable by Lemma~\ref{lem: KPW}. 

For sub-case (b), $m=3$, and we let $S= \{u,v_{3,1},v_{3,2}, v_{3,3}, v_{3,4}\}$. Note that $\Delta (G-S) \leq 4$. So, by Theorem~\ref{thm: KKresult} we know that $G-S$ has an equitable $L$-coloring.  We name the vertices of $S$ as follows: $x_1 =u$, $x_2=v_{3,1}$, $x_3=v_{3,4}$, $x_4=v_{3,3}$, and $x_5=v_{3,2}$. Note $|N_G(x_1)-S| \leq 4$, $|N_G(x_2)-S| = 2$, $|N_G(x_3)-S| \leq 2$, $|N_G(x_4)-S| \leq 1$, and $|N_G(x_5)-S| = 0$. So, by Lemma~\ref{lem: KPW} we know that $G$ has an equitable $L$-coloring. 
\end{proof}

\begin{lem}\label{lem: l_m 6}
Suppose $m \geq 3$.  Suppose $H=\Theta(l_1, l_2, \ldots, l_m)$ where $l_1 \geq 2$, $l_2 \geq 4$, and $l_m \geq 6$. If $G = H^2$, then $G$ is equitably $(m+2)$-choosable.
\end{lem}

\begin{proof}
Let $L$ be an arbitrary $(m+2)$-assignment for $G$. We will show that $G$ has an equitable $L$-coloring.  We let 
$$S=\{u,v_{m,2},v_{m,3}, v_{m,4}\} \cup \{v_{i,1}: 3 \leq i \leq m\}.$$ 
There exist natural numbers $a_1, \ldots, a_m$ satisfying $a_1 \leq \cdots \leq a_m$ and $a_m \geq 3$ such that the following holds.  If $H = [B(a_1, \ldots, a_m)]^2$, then there are $i, j \in [m]$ satisfying $1 \leq i < j \leq m$ such that $G-S$ is $H$ plus an edge between the vertices $v_{i,a_i}$ and $v_{j,a_j}$.  Note that we know such an $i$ and $j$ exist since $v_{1,1}$ and $v_{2,1}$ are adjacent in $G-S$. Also note that we know $a_m \geq 3$ since $v_{2,1}$, $v_{2,2}$, and $v_{2,3}$ are in $G-S$. By Lemma~\ref{lem: extra edge} we know that $G-S$ has an equitable $L$-coloring. We name the vertices of $S$: $x_1, x_2, \ldots , x_{m+2}$ where $x_1 = u$, $x_{m-1} = v_{m,4}$, $x_m = v_{m,1}$, $x_{m+1} = v_{m,3}$ and $x_{m+2}= v_{m,2}$. Finally, if $m \geq 4$, we arbitrarily name the remaining vertices in $S$: $x_2, \ldots , x_{m-2}$ in an injective fashion. Note that $|N_G(x_1) -S| = m+1$, $|N_G(x_{m-1})-S| = 2 $, $|N_G(x_m)-S|=1$, $|N_G(x_{m+1})-S| =1$, $|N_G(x_{m+2})-S| =0$, and $|N_G(x_i)-S|\leq 4 \leq m+2-i$ for all $2\leq i \leq m-2$. Thus, $G$ has an equitable $L$-coloring by Lemma~\ref{lem: KPW}.
\end{proof}

{\bf Acknowledgment.} The authors would like to thank Hemanshu Kaul and Michael Pelsmajer for their helpful comments on this paper.

\end{document}